\documentclass{article}

\usepackage[english]{babel}
\usepackage{hyperref}
\usepackage{amsfonts,amsthm,amssymb}

\newtheorem{example}{Example}[section]
\newtheorem{remark}{Remark}[section]
\newtheorem{definition}{Definition}[section]
\newtheorem{proposition}{Proposition}[section]
\newtheorem{theorem}{Theorem}[section]
\newtheorem{lemma}{Lemma}[section]

\newcommand{\R}{\mathbb{R}}

\begin{document}

\title{Necessary Optimality Conditions for Continuous-Time Optimization Problems with Equality and Inequality Constraints}

\author{
     M. R. C. MONTE%
     \thanks{moises.monte@ufu.br}
     \and
     V. A. de OLIVEIRA%
     \thanks{antunes@ibilce.unesp.br.}
}

\date{Instituto de Bioci\^encias, Letras e Ci\^encias Exatas, \\ UNESP -- Universidade Estadual Paulista, \\ S\~ao Jos\'e do Rio Preto, SP, Brasil}

\maketitle

\begin{abstract}

The paper is devoted to obtain first and second order necessary optimality conditions for continuous-time optimization problems with equality and inequality constraints. A full rank type regularity condition along with an uniform implicit function theorem are used in order to establish such necessary conditions.


\end{abstract}

\section{Introduction}

We are concerned with the general nonlinear continuous-time optimization problem with equality and inequality constraints in the form
\begin{equation}\label{1.1}
\begin{array}{ll} \mbox{maximize} & P(z)=\displaystyle\int_{0}^{T}\phi(z(t), t) dt \\ 
\mbox{subject to} & h(z(t),t)=0 ~~ \mbox{a.e.} ~t\in [0,T],\\
& g(z(t),t)\geq 0 ~~ \mbox{a.e.} ~ t\in [0,T], \\
& z\in L_{\infty}([0,T];\R^{n}),
\end{array}
\end{equation}
where $\phi: \R^{n}\times[0,T]\rightarrow\R$, $h:\R^{n}\times[0,T]\rightarrow\R^{p}$ and $g:\R^{n}\times [0,T]\rightarrow\R^{m}$. All vectors are column vectors, unless transposed when they will be denoted by a prime, and all integrals are in the Lebesgue sense.

Continuous-time problems arise often in the literature and were first proposed by Bellman \cite{bellman:1953,bellman:1957} in his studies of some dynamical models of production and inventory called ``bottleneck processes'', which gave rise to continuous-time linear programming. Such problems can be posed as
\[
\begin{array}{ll} \mbox{maximize} & P(z)=\displaystyle\int_{0}^{T}a'z(t) dt \\ 
\mbox{subject to} & z(t)\geq 0, ~~ 0\leq t\leq T,\\
& Bz(t)\leq c+\displaystyle\int_{0}^{t}Kz(s)~ds, ~~ 0\leq t\leq T, \\
& z \in L_{\infty}([0,T];\R^{n}),
\end{array}
\]
where $B$ and $K$ are $m \times n$ matrices, $a$ is a $n$-vector and $c$ is a $m$-vector. Considering a certain dynamic generalization of an ordinary linear programming problem, he formulated a corresponding dual problem, established a weak duality theorem, and suggested some computational procedures. Subsequently, Bellman's formulation and duality theory were substantially extended to more general forms of continuous-time linear programming problems, and also to certain classes of continuous-time nonlinear programming problems. For a summary of the results pertaining to duality theory in continuous-time programming and a fairly extensive list of relevant references the reader is referred to Zalmai \cite{zalmai:1985}.

Optimality conditions of the Karush-Kuhn-Tucker type were first considered in continuous-time programming by Hanson and Mond \cite{hanson:1968} for the following linearly constrained nonlinear program:
\[
\begin{array}{ll} \mbox{maximize} & P(z)=\displaystyle\int_{0}^{T}\phi(z(t)) dt \\ 
\mbox{subject to} & z(t)\geq 0, ~~ 0\leq t\leq T,\\
& B(t)z(t)\leq c(t)+\displaystyle\int_{0}^{t}K(t,s)z(s)~ds, ~~ 0\leq t\leq T, \\
& z \in L_{\infty}([0,T];\R^{n}),
\end{array}
\]
where $B(t)$ is an $m\times n$ matrix piece-wise continuous on $[0,T]$, $c(t) \in \R^{n}$ is piece-wise continuous on $[0,T]$, $K(s,t)$ is a $m\times n$ matrix piece-wise continuous on $[0,T]\times [0,T]$ and $\phi$ is a given concave scalar function twice continuously differentiable. For this purpose, certain positivity conditions on $B(t)$, $c(t)$ and $K(s, t)$ were imposed and the objective function was linearized in order to apply a extended version of Levinson's linear duality result \cite{levinson:1966}. A duality theorem for the nonlinear problem under consideration in then established, and the Karush-Kuhn-Tucker conditions are deduced as a consequence of this nonlinear duality theorem. In this way, many other authors obtained necessary and sufficient conditions for continuous-time problems with non-linear inequality constraints, for example, Farr and Hanson \cite{farr:1974,farr1:1974}, Reiland and Hanson \cite{reiland:1980}.  

Roughly speaking, optimality conditions for continuous-time nonlinear programming problems have been obtained by direct methods. In Abrham and Buie \cite{abrham:1979} a certain regularity assumption is used to establish the Kuhn-Tucker conditions for a class of convex programming problems. Reiland \cite{reiland1:1980}, employing a continuous-time version of Zangwill's constraint qualification \cite{zangwill:1969} introduced in \cite{reiland:1980}, and an infinite-dimensional form of Farkas' theorem \cite{craven:1977}, established optimality conditions and duality relations for differentiable continuous-time programs.

Brandão, Rojas-Medar and Silva, tackled nonsmooth continuous-time optimization problems, first in \cite{rojas:1998} where sufficient conditions were obtained and then in \cite{brandao:2001} which refers to necessary conditions. In \cite{de:2007}, de Oliveira and Rojas-Medar generalized concepts of the KKT-invexity and WD-invexity introduzed by Martin \cite{martin:1985} for mathematical programming problems, proving that the notion of KKT-invexity is a necessary and sufficient condition for global optimality of a Karush-Kuhn-Tucker point and that the notion of WD-invexity is a necessary and sufficient condition for weak duality. The same authors established KKT-invexity for nonsmooth continuous-time programming problems \cite{antunes:2007}. The multiobjective case was considered in \cite{de1:2007}. de Oliveira \cite{de:2010} also studied multiobjective continuous-time programming problems, but without imposing any differentiability assumption. Saddlepoint type optimality conditions, duality theorems as well as results on the scalarization method  were presented. The concept of pre-invexity was utilized.

In the formulation given in (\ref{1.1}), where equality and inequality constraints are present and the feasible solutions belong to $L_{\infty}([0,T];\R^{n})$, necessary optimality conditions are not found in the literature. We believe this is due to the fact that even a few years ago, a crucial tool for the treatment of equality constraints was not available: the uniform implicit function theorem. When fixing $t$, we can apply the classical one. But the implicit function thereby obtained may not have good properties, such as measurebility for example, with respect to $t$. Such result only appear in $1997$ in paper by Pinho and Vinter \cite{pinho:1997}. On the other hand, in $L_{\infty}([0,T];\R^{n})$ and with inequality constraints only, there is a vast literature, as cited above. In the case of formulations in other spaces, we cite Zalmai \cite{zalmai:1998}, for instance, where the feasible solutions are in a Hilbert space.

In this work, by means of the uniform implicit function theorem presented by Pinho and Vinter \cite{pinho:1997} and the use of a full rank type condition, we obtain first and second order necessary optimality conditions for continuous-time programming problems with equality and inequality constraints. The paper is organized as follows. In Section 2, we give some preliminaries. In Section 3, we consider problems with equality constraints only. Finally, in Section 4, the general case is treated. 

\section{Preliminaries}

We denote by 
$$
\Omega=\{z\in L_{\infty}([0,T];\R^{n})\mid h(z(t),t)=0,~g(z(t),t)\geq 0~\mbox{a.e.}~ t \in[0,T]\}
$$
the feasible set of problem (\ref{1.1}). By simplicity, given $\bar{z} \in \Omega$, we will write
$$
\bar{\phi}(t)=\phi(\bar{z}(t),t)~~\mbox{and}~~~ \nabla\bar{\phi}(t)=\nabla \phi(\bar{z}(t),t)~~~\mbox{a.e.~} t \in [0,T]
$$ 
as well as for $h,~\nabla h,~g,~\nabla g$ and its components. Set index sets $I=\{1,\ldots, p\}$ and $J=\{1,\ldots,m\}$ and define, a.e. $t\in [0,T]$, the index set of all binding constraints at $\bar{z}\in\Omega$ as 
$$
I_{a}(t)=\{j\in J\mid \bar{g}_{j}(t)=0\},
$$ 
and $I_{c}(t) = [0,T] \setminus I_a(t)$, its complement. For a.e. $t \in [0,T]$, let $q_{a}(t)$ and $q_{c}(t)$ the cardinals of the $I_{a}(t)$ and $I_{c}(t)$, respectively. Denote by $\delta P(z;\gamma)$ the Fréchet derivative of $P$ at $z$ with increment $\gamma\in L_{\infty}([0,T];\R^{n})$.

\begin{definition} We say that $\bar{z}\in \Omega$ is a local optimal solution of (\ref{1.1}) if there exists $\epsilon>0$ such that $P(\bar{z})\geq P(z)$ for all $z\in \Omega$ satisfying $z \in \bar{z} + \epsilon\bar{B}$, where $\bar{B}$ denotes the closed unit ball with center at the origin.
\end{definition}

Reiland (\cite{reiland1:1980}, Theorem 2) provide a condition for obtaining directions that, starting at a given point $z$, the values of objective function are increased. 

\begin{proposition}\label{Reiland}(\cite{reiland1:1980}) If 
$$\displaystyle\int_{0}^{T} \nabla\phi'(z(t),t)\gamma(t)dt>0$$
where $z,\gamma\in L_{\infty}([0,T];\R^{n})$, then there exists a number $\sigma>0$ such that $$P(z+\tau\gamma)>P(z)~\mbox{for}~0<\tau\leq\sigma.$$
\end{proposition}

Let $\{F_{a}:\R^{n}\rightarrow\R^{n}\mid a\in A\}$ be a family of maps parameterized by points $a$ in a subset $A\subset\R^{n}$. If $\nabla F_{a}$ is nonsingular at some point $x_{0}$ for all $a\in A$, we know by the classic inverse mapping theorem that, for each $a$, there exists some neighborhood of $x_{0}$ on which $F_{a}$ is smoothly invertible. The following uniform inverse mapping theorem (\cite{pinho:1997}, Proposition 4.1) and, consequently, the uniform implicit function theorem (\cite{pinho:1997}, Corollary 4.2), that will have important roles in the proof of the results of Sections 3 and 4, give conditions under which the same neighborhood of $x_{0}$ can be chosen for all $a\in A$.

\begin{proposition}[Uniform Implicit Function Theorem, \cite{pinho:1997}]\label{TFIU} Consider a set $A\subset\R^{k}$, a number $\alpha > 0$, a family of functions $$\{\psi_{a}:\R^{m}\times\R^{n}\rightarrow\R^{n}\}_{a\in A},$$ and a point $(u_{0},v_{0})\in\R^{m}\times\R^{n}$ such that $\psi_{a}(u_{0},v_{0})=0$ for all $a\in A$. Assume that:
\begin{itemize}
\item[(i)] $\psi_{a}$ is continuously differentiable on $(u_{0},v_{0})+\alpha B$, uniformly in $a\in A$;
\item[(ii)] there exists a monotone increasing function $\theta:(0,\infty)\rightarrow (0,\infty)$, with $\theta (s)\downarrow 0$ as $s\downarrow 0$, such that 
\[
\|\nabla\psi_{a}(\tilde{u},\tilde{v})-\nabla\psi_{a}(u,v)\|\leq\theta(\|(\tilde{u},\tilde{v})-(u,v)\|)
\] 
for all $a\in A$, $(\tilde{u},\tilde{v}),~(u,v)\in (u_{0},v_{0})+\alpha B$;
\item[(iii)] $\nabla_{v}\psi_{a}(u_{0},v_{0})$ is nonsingular for each $a\in A$ and exists $c>0$ such that 
\[\|[\nabla_{v}\psi_{a}(u_{0},v_{0})]^{-1}\|\leq c~~\mbox{for all}~a\in A.
\]
\end{itemize} 
Then there exist $\delta \geq 0$ and a family of continuously differentiable functions \[\{\phi_{a}:u_{0}+\delta B\rightarrow v_{0}+\alpha B\}_{a\in A}\] which are Lipschitz continuous with a common Lipschitz constant $K$ such that
\begin{eqnarray}\nonumber
v_{0} &=& \phi_{a}(u_{0})~~\forall~a\in A,\\ \nonumber
\psi_{a}(u,\phi_{a}(u)) &=& 0,~~\forall~u\in u_{0}+\delta B~\forall~a\in A,~~\mbox{and}\\ \nonumber
\nabla_{u}\phi_{a}(u_{0}) &=& -[\nabla_{v}\psi_{a}(u_{0},v_{0})]^{-1}\nabla_{u}\psi_{a}(u_{0},v_{0}).
\end{eqnarray}
The numbers $\delta$ and $K$ depend only on $\theta(\cdot)$, $c$ and $\alpha$. Furthermore, if $A$ is a Borel set and $a\mapsto\psi_{a}(u,v)$ is a measurable Borel function for each $(u,v)\in (u_{0},v_{0})+\alpha B$, then $a\mapsto\phi_{a}(u)$ is a measurable Borel function for each $u\in u_{0}+\epsilon B$.
\end{proposition}

\section{KKT conditions for problems with equality constraints}

In this section, we will consider the continuous-time programming problem with equality constraints only. The general case is postponed to the next section. We start with the continuous-time problem without constraints, instead. The necessary optimality conditions for unrestricted problems will be used later in the proof of the main result of this section.

Consider the unrestricted continuous-time problem, namely,
\begin{equation}\label{SR}
\begin{array}{ll} \mbox{maximize} & P(z)=\displaystyle\int_{0}^{T}\phi(z(t), t) dt \\
\mbox{subject to} & z\in \Omega =L_{\infty}([0,T];\R^{n}).
\end{array}
\end{equation} 
Assume that
\begin{itemize}
\item[(H1)] $\phi(\cdot,t)$ is twice continuously differentiable throughout $[0,T]$; $\phi(z,\cdot)$ is measurable for each $z$ and there exists $K_{\phi}>0$ such that \[\|\nabla\phi(\bar{z}(t),t)\|\leq K_{\phi}~~\mbox{a.e. }~ t \in [0,T].\] 
\end{itemize}

\begin{proposition}\label{IL-SR} 
If $\bar{z}$ is a local optimal solution for $(\ref{SR})$, then 
$$
\nabla\bar{\phi}(t)=0~~\mbox{a.e.}~ t \in [0,T].
$$ 
and 
\[
\displaystyle\int_{0}^{T}\gamma'(t)\nabla^{2}\bar{\phi}(t)\gamma(t)~dt\leq 0~~\forall~\gamma\in L_{\infty}([0,T];\R^{n}).
\]
\end{proposition}
\begin{proof}
Let $\gamma\in L_{\infty}([0,T];\R^{n})$. From the local optimality of $\bar{z}$, there exists $\bar{\tau} > 0$ such that $P(\bar{z}) \geq P(\bar{z}+\tau\gamma)$ for all $\tau \in (0,\bar{\tau})$. 

By first order Taylor expansion in Banach spaces \cite{lusternik:1961} we have that 
\[
0\geq P(\bar{z}+\tau\gamma)-P(\bar{z})=\tau\delta P(\bar{z};\gamma)+\varepsilon(\tau),
\] 
where $(\varepsilon(\tau)/\tau)\rightarrow 0$ when $\tau\rightarrow 0$. Dividing both sides by $\tau>0$ and taking limits as $\tau \downarrow 0$ we have that $\delta P(\bar{z};\gamma)\leq 0$. Similarly, $\delta P(\bar{z};-\gamma) \leq 0$. Therefore, $\delta P(\bar{z};\gamma)=0$, that is,  
\[
\displaystyle\int_{0}^{T}\nabla\phi'(\bar{z}(t),t)\gamma(t)~dt=0~\forall~\gamma\in L_{\infty}([0,T];\R^{n}).
\] 
From the last equality we see that $\nabla\phi(\bar{z}(t),t)=0$ a.e. $t \in [0,T]$.

By the second order Taylor expansion, we can write 
\begin{eqnarray}\nonumber
0\geq P(z+\tau\gamma)-P(\bar{z}) &=& \tau\delta P(\bar{z};\gamma)+\frac{1}{2}\tau^{2}\delta^{2}P(\bar{z};(\gamma,\gamma))+\varepsilon(\tau)\\ \nonumber
&=& \frac{1}{2}\tau^{2}\delta^{2}P(\bar{z};(\gamma,\gamma))+\varepsilon(\tau),
\end{eqnarray}
where $(\varepsilon(\tau)/\tau^2)\rightarrow 0$ when $\tau\rightarrow 0$. Dividing both sides by $\tau^2>0$ and taking limits as $\tau \to 0$, we obtain
\[
\frac{1}{2}\delta^{2}P(\bar{z};(\gamma,\gamma)) \leq 0 \Leftrightarrow \displaystyle\int_{0}^{T}\gamma'(t)\nabla^{2}\phi(\bar{z}(t),t)\gamma(t)dt\leq 0.
\] 
The proof is complete.
\end{proof}

Now, consider the continuous-time problem with equality constraints:
\begin{equation}\label{3.2}
\begin{array}{ll} \mbox{maximize} & P(z)=\displaystyle\int_{0}^{T}\phi(z(t), t) dt \\ 
\mbox{subject to} & h(z(t),t) = 0 ~\mbox{a.e.}~ t \in [0,T].
\end{array}
\end{equation}
In this case, 
\[
\Omega=\{z\in L_{\infty}([0,T];\R^{n})\mid h(z(t),t)=0~\mbox{a.e.}~ t \in [0,T]\}.
\] 
Let us remember that $\phi:\R^{n}\times [0,T]\rightarrow\R$ and $h:\R^{n}\times[0,T]\rightarrow\R^{p},~p\leq n$. Given $\epsilon>0$ and $\bar{z}\in\Omega$, we assume that, in addition to (H1), the following hypothesis are valid:
\begin{itemize}
\item[(H2)] $h(z,\cdot)$ is measurable for each $z$ and $h(\cdot,t)$ is twice continuously differentiable on $\bar{z}(t)+\epsilon \bar{B}$ for a.e. $t \in [0,T]$.
\item[(H3)] There exists an increasing function $\tilde{\theta}:(0,\infty)\rightarrow (0,\infty)$, $\theta(s)\downarrow 0$ as $s\downarrow 0$, such that for all $\tilde{z},z\in \bar{z}(t)+\epsilon \bar{B}$ and $\mbox{a.e.}~t \in [0,T]$, 
\[
\|\nabla h(\tilde{z},t)-\nabla h(z,t)\| \leq \tilde{\theta}(\|\tilde{z}-z\|).
\] 
There exists $K_{0}>0$ such that for a.e. $t \in [0,T]$, 
$$
\|\nabla h(\bar{z}(t),t)\| \leq K_{0}.
$$
\item[(H4)] There exists $K>0$ such that for a.e. $t \in [0,T]$, 
$$
\det\{\nabla \bar{h}(t) \nabla \bar{h}^{\prime}(t)\} \geq K,
$$ 
where $\nabla \bar{h}(t) = \nabla h(\bar{z}(t),t)$.  
\end{itemize}

\begin{remark} 
Hyphotesis (H4) guarantees that the rows of $\nabla \bar{h}(t)$, formed by the gradient vectors $\nabla \bar{h}_{i}(t)$, $i\in I$, are linearly independent for almost every $t\in [0,T]$. Moreover, along with (H3), it guarantees also that the norm of $[\nabla \bar{h}(t) \nabla \bar{h}^{\prime}(t)]^{-1}$ is uniformly bounded, a required property in the application of the uniform implicit theorem. See proposition below. 
\end{remark}

\begin{proposition}\label{Inv Ltda} 
Consider a subset $A\subset\R^{k}$ and $\{M_{a}\}_{a\in A}$ a family of $p\times p$ matrices such that 
$$
det(M_{a}) \geq K,~a\in A, \quad \mbox{and} \quad \|M_{a}\|\leq L,~a\in A,
$$ 
for some $K,L>0$. Then there exists $C>0$ such that
\[
\|[M_{a}]^{-1}\| \leq C,~a\in A.
\]
\end{proposition}
\begin{proof}
Consider the singular values decomposition 
\[
M_{a}=U_{a}\Sigma_{a}V^{-1}_{a},~a\in A,
\] 
where $U_{a}$ and $V_{a}$ are $p \times p$ unit matrices for all $a\in A$ and $\Sigma_{a}=\mbox{diag}\{\sigma^{a}_{i}\}_{i=1}^{p}$ are diagonal matrices with singular values ordered, without loss of generality, in decreasing order 
\[
\sigma_{1}^{a}\geq\sigma_{2}^{a}\geq \ldots\sigma_{p}^{a}>0,~a\in A.
\] 
Thus,  
\[
L \geq \|M_{a}\|=  \|U_{a}\Sigma_{a}V^{-1}_{a}\| = \|\Sigma_{a}\|,~a\in A,
\] 
so that 
\[
\sigma_{i}^{a} \leq \displaystyle\max_{i\in I}\sigma_{i}^{a} = \|\Sigma_{a}\| \leq L,~a\in A,~i\in I,
\] 
which, in turn, imply that 
\[
\prod_{i=1}^{p-1}\sigma_{i}^{a}\leq L^{p-1},~a\in A.
\]
On the other hand, 
\[
\det(M_{a})=\prod_{i=1}^{p}\sigma_{i}^{a}\geq K,~a\in A \Leftrightarrow \sigma_{p}^{a} \geq K\left[\prod_{i=1}^{p-1}\sigma_{i}^{a}\right]^{-1} \geq \frac{K}{L^{p-1}},~a\in A.
\]
Therefore, 
\begin{eqnarray*}
\| [M_{a}]^{-1} \| & = & \| [V_{a}\Sigma^{-1}_{a}U_{a}^{-1}] \| = \| \Sigma^{-1}_{a} \| \\
& = & \max_{i\in I} \left\{ \frac{1}{\sigma_{i}^{a}} \right\} = \frac{1}{\sigma_{p}^{a}} \leq \frac{L^{p-1}}{K},~a\in A,
\end{eqnarray*}
which concludes the proof with $C=L^{p-1}/K$. 
\end{proof}

We are now in position to state and prove the main result of the section. In the sequel, Karush-Kuhn-Tucker type necessary optimality conditions of first and second order are provided for problem (\ref{3.2}) under the full rank condition (H4).

\begin{theorem}\label{TKKT-EC} 
Let $\bar{z}$ be a local optimal solution for (\ref{3.2}) and suppose that (H1)-(H4) do hold. Then, there exists $u\in L_{\infty}([0,T];\R^{p})$ such that 
\begin{eqnarray}
\label{3.3} \nabla\bar{\phi}(t)+\displaystyle\sum_{i=1}^{p}u_{i}(t)\nabla \bar{h}_{i}(t)=0~\mbox{a.e.}~ t \in [0,T],
\end{eqnarray}
and
\begin{eqnarray} \label{3.4} 
\int_{0}^{T}\gamma'(t)\lbrace \nabla^{2}\bar{\phi}(t)+\sum_{i=1}^{p}u_{i}(t)\nabla^{2}\bar{h}_{i}(t)\rbrace\gamma(t)~dt\leq 0~~\forall~\gamma\in N,
\end{eqnarray}
where $N$ is given by
\[N=\{\gamma\in L_{\infty}([0,T];\R^{n})\mid\nabla \bar{h}(t)\gamma(t)=0~~\mbox{a.e.}~t \in [0,T]\}.\] 
\end{theorem}
\begin{proof} 
Let $\bar{z}$ be a local optimal solution of (\ref{3.2}) on $\bar{z}+\epsilon \bar{B}$. The proof is divided in several steps.

\noindent\textbf{STEP 1:} We define an application that satisfy the conditions of Proposition \ref{TFIU}. Let $S_{0} \subset [0,T]$ be the largest subset where each of the conditions in (H1)-(H4) do not hold for every $t \in S_{0}$. We know from the assumptions that $S_{0}$ has Lebesgue measure equal to zero. It follows from (\cite{rudin:1976},~p. 309) that there exists a Borel set $S$, which is the intersection of a countable collection of open sets, such that $S_{0} \subset S$ and $S \setminus S_{0}$ has Lebesgue measure equal to zero. Thence $S$ is a Borel set which has Lebesgue measure equal to zero, so that $[0,T] \setminus S$ has full measure. In Proposition \ref{TFIU}, identify the Borel set $[0,T] \setminus S$ with $A$, $t$ with $a$, $(\xi,\eta)$ with $(u, v)$ and $(0,0)$ with $(u_{0},v_{0})$.

Define $\mu :\R^{n} \times \R^{p} \times [0,T] \rightarrow \R^{p}$ as 
$$
\mu (\xi,\eta,t)=h(\bar{z}(t)+\xi+\nabla \bar{h}^\prime(t)\eta,t).
$$ 
Let us check that the assumptions of Proposition \ref{TFIU} are fulfilled. First note that setting $\alpha=\min\{\frac{\epsilon}{2},\frac{\epsilon}{2K_{0}}\}$, we have that
\[
\| \bar{z}(t)+\xi+\nabla \bar{h}^\prime(t)\eta-\bar{z}(t) \| = \| \xi+\nabla \bar{h}^\prime(t)\eta \| \leq \| \xi \| + \| \nabla \bar{h}^\prime(t) \| \cdot \| \eta || \leq \epsilon,
\] 
whenever $(\xi,\eta)\in (0,0)+\alpha\bar{B}$. We have also that 
$$
\mu(0,0,t)=h(\bar{z},t)=0 ~~ \mbox{a.e.} ~ t \in [0,T]. 
$$
Let $(\tilde{\xi},\tilde{\eta}),~(\xi,\eta)\in (0,0)+\alpha B$, $t\in A$. Then,
\begin{eqnarray*}
&& \| \nabla \mu(\tilde{\xi},\tilde{\eta},t) - \nabla \mu(\xi,\eta,t) \| \\ 
&& \quad = \| [ \nabla h(\bar{z}(t)+\tilde{\xi}+\nabla \bar{h}^\prime(t)\tilde{\eta},t) ~~ \nabla h(\bar{z}(t)+\tilde{\xi}+\nabla \bar{h}^\prime(t)\tilde{\eta},t) \nabla \bar{h}^\prime(t)] \\
&& \quad\qquad-  [ \nabla h(\bar{z}(t)+\xi+\nabla \bar{h}^\prime(t)\eta,t) ~~ \nabla h(\bar{z}(t)+\xi+\nabla \bar{h}^\prime(t)\eta,t)\nabla \bar{h}^\prime(t)] \| \\
&& \quad = \| (\nabla h(\bar{z}(t)+\tilde{\xi}+\nabla \bar{h}^\prime(t)\tilde{\eta},t) - \nabla h(\bar{z}(t)+\xi+\nabla \bar{h}^\prime(t)\eta,t))[I_{n} ~~ \nabla \bar{h}^\prime(t)] \| \\
&& \quad \leq \| \nabla h(\bar{z}(t)+\tilde{\xi}+\nabla \bar{h}^\prime(t)\tilde{\eta},t) - \nabla h(\bar{z}(t)+\xi+\nabla \bar{h}^\prime(t)\eta,t) \| \cdot \| [I_{n} ~~ \nabla \bar{h}^\prime(t) \| \\
&& \quad \leq \tilde{\theta}(\|(\tilde{\xi}-\xi)+\nabla \bar{h}^\prime(t)(\tilde{\eta}-\eta)\|) \cdot (1+K_{0}) \\ 
&& \quad \leq \tilde{\theta}(\|(\tilde{\xi}-\xi)\| + K_{0}\|(\tilde{\eta}-\eta)\|) \cdot (1+K_{0}) \\
&& \quad \leq \tilde{\theta}(\|(\tilde{\xi}-\xi,\tilde{\eta}-\eta)\| + K_{0} \cdot \|(\tilde{\xi}-\xi,\tilde{\eta}-\eta)\|) \cdot (1+K_{0}) \\
&& \quad = \theta(\|(\tilde{\xi},\tilde{\eta})-(\xi,\eta)\|),
\end{eqnarray*}
where $\theta:(0,\infty)\rightarrow (0,\infty)$, $\theta(s)=(1+K_{0})\tilde{\theta}(s+K_{0}s)$, is an increasing monotone function such that $\theta(s)\downarrow 0$ when $s\downarrow 0$. We have that 
$$
\nabla_{\eta}\mu(0,0,t) = \nabla \bar{h}(t) \nabla \bar{h}^\prime(t)~~\mbox{a.e.}~ t \in [0,T].
$$ 
Thus, by assumption (H4), $\nabla_{\eta}\mu(0,0,t)$ is nonsingular for each $t\in A$. By making use of (H3), it follows from Proposition \ref{Inv Ltda} that there exists $M>0$ such that 
\begin{eqnarray} \label{3.5} 
\| [\nabla \bar{h}(t) \nabla \bar{h}^\prime(t)]^{-1} \| \leq M~\mbox{a.e.}~ t \in [0,T].
\end{eqnarray}
By Proposition \ref{TFIU} there exist $\sigma \in (0,\epsilon)$, $\delta \in (0,\epsilon)$ and an implicit function $d : \sigma B \times A \rightarrow\delta B$ such that $d(\xi,\cdot)$ is measurable for fixed $\xi$, the functions of family $\{d(\cdot,t) \mid t \in A\}$ are Lipschitz continuous with a common Lipschitz constant, $d(\cdot,t)$ is continuously differentiable for each $t \in A$, and for a.e. $t \in [0,T]$,
\begin{eqnarray}
\label{3.6} d(0,t) & = & 0,\\
\label{3.7} \mu(\xi,d(\xi,t),t) & = & 0,~ \xi \in\sigma B, \\
\label{3.8} \nabla d(0,t) & = & -[\nabla \bar{h}(t) \nabla \bar{h}^\prime(t)]^{-1}\nabla \bar{h}(t).
\end{eqnarray}
Choose $\sigma_{1}>0$ and $\delta_{1}>0$ such that
\begin{eqnarray}
\label{3.9} \sigma_{1}\in (0,\min\{\sigma,\frac{\epsilon}{2}\}),\hspace{1cm}\delta_{1}\in (0,\min\{\delta,\frac{\epsilon}{2}\}),\hspace{1cm}\sigma_{1}+K_{0}\delta_{1}\in (0,\frac{\epsilon}{2}),
\end{eqnarray}
where $K_{0}$ is given by (H3). In the following steps and without loss of generality, we consider the implicit function $d$ defined on $\sigma_{1} B\times [0,T]$ and taking values in $\delta_{1}B$.

\noindent\textbf{STEP 2:} We show that if $\bar{z}$ is a local optimal solution of (\ref{3.2}), then it is a local optimal solution of the following auxiliary problem
\begin{equation}\label{3.10}
\begin{array}{ll} \mbox{maximize} & \tilde{P}(z)=\displaystyle\int_{0}^{T}\varphi(z(t),t) dt \\ 
\mbox{subject to} & z \in L_{\infty}([0,T];\R^{n}),
\end{array}
\end{equation}
where $\varphi(z(t),t) = \phi(z(t)+\nabla \bar{h}^\prime(t) d(z(t)-\bar{z}(t),t),t)$. Indeed, suppose that $\tilde{z} \in \bar{z}+\sigma_{2}B$, for arbitrary $0<\sigma_{2}<\sigma_{1}$, is a feasible solution of problem (\ref{3.10}) such that $\tilde{P}(\tilde{z}) > \tilde{P}(\bar{z})$. Consider  
$$
\hat{z}(t) = \tilde{z}(t)+\nabla \bar{h}^\prime(t) d(\tilde{z}(t)-\bar{z}(t),t) ~\mbox{a.e.}~ t \in [0,T].
$$ 
Using (\ref{3.9}) and (H3), we have that
\begin{eqnarray*}
\| \hat{z}(t)-\bar{z}(t) \| & = & \| (\tilde{z}(t)-\bar{z}(t))+\nabla \bar{h}^\prime(t) d(\tilde{z}(t)-\bar{z}(t),t)\| \\ 
&\leq & \| \tilde{z}(t)-\bar{z}(t) \| + \| \nabla \bar{h}^\prime(t) \| \cdot \| d(\tilde{z}(t)-\bar{z}(t),t) \| < \sigma_{1}+K_{0}\delta_{1} < \epsilon.
\end{eqnarray*} 
As $\tilde{z}-\bar{z} \in \sigma_{1}B$, using the definition of $\mu$ we have that for a.e. $t \in [0,T]$,
$$
\mu(\tilde{z}(t)-\bar{z}(t),d(\tilde{z}(t)-\bar{z}(t),t),t) = 0 \Rightarrow h(\tilde{z}(t)+\nabla \bar{h}^\prime(t) d(\tilde{z}(t)-\bar{z}(t),t),t)=0,
$$ 
that is, $h(\hat{z}(t),t)=0$ a.e. $t \in [0,T]$. But, 
$$
P(\hat{z}) = \tilde{P}(\tilde{z}) > \tilde{P}(\bar{z}) = P(\bar{z}),
$$ 
contradicting the fact that $\bar{z}$ is a local optimal solution of (\ref{3.2}).

\noindent\textbf{STEP 3:} Applying  Proposition \ref{IL-SR}, we have that for a.e. $t \in [0,T]$,
\begin{eqnarray*}
0 & = & \nabla\varphi(\bar{z}(t),t) \\ 
& =& \lbrace I_{n}+\nabla \bar{h}^\prime(t) \nabla d(0,t)\rbrace^\prime \nabla \phi(\bar{z}(t) + \nabla \bar{h}^\prime(t) d(0,t),t) \\
& = & \nabla \phi(\bar{z}(t),t) + \nabla d^\prime(0,t) \nabla \bar{h}(t) \nabla\phi(\bar{z}(t),t) \\
& = & \nabla \phi(\bar{z}(t),t) + \nabla h^\prime(\bar{z}(t),t) \lbrace -[ \nabla \bar{h}(t) \nabla \bar{h}^\prime(t) ]^{-1} \nabla \bar{h}(t) \nabla\phi(\bar{z}(t),t)\rbrace \\
& = & \nabla \phi(\bar{z}(t),t) + \nabla h^\prime(\bar{z}(t),t)u(t) \\
& = & \nabla \phi(\bar{z}(t),t) + \sum_{i=1}^{p} u_{i}(t) \nabla h_{i}(\bar{z}(t),t)
\end{eqnarray*}
where 
$$
u(t)=-[ \nabla \bar{h}(t) \nabla \bar{h}^\prime(t) ]^{-1} \nabla \bar{h}(t) \nabla\phi(\bar{z}(t),t)~\mbox{a.e.}~ t \in [0,T].
$$ 
Observe that $u\in L_{\infty}([0,T];\R^{p})$ is unique and that 
$$
\| u(t) \| \leq MK_{0}K_{\phi}~~\mbox{a.e.}~ t \in [0,T],
$$
by hypotheses (H1) and (H3) and by (\ref{3.5}). 

Now, being $\phi(\cdot,t)$ and $h(\cdot,t)$ twice continuously differentiable on $\bar{z}+\epsilon \bar{B}$ throughout $[0,T]$, it follows directly from its definition that $\mu(\xi,d(\xi,t),t)$ is twice continuously differentiable on $\sigma_{1} B$ for a.e. $t \in [0,T]$. Consequently, from Corollary \ref{TFIU}, $d$ is continuously differentiable in a neighborhood of $\xi=0$ (by simplicity, consider this neighborhood as being $\sigma_{1}B$). By Proposition \ref{IL-SR} we have that 
$$
\int_{0}^{T}\gamma^\prime(t)\nabla^{2}\varphi(z(t),t)\gamma(t)~dt\leq 0~~\forall~\gamma\in L_{\infty}([0,T];\R^{n}).
$$ 
Let us calculate $\nabla^2\varphi$. We have that
\begin{eqnarray*}
\nabla\varphi(z(t),t) & = & [I_{n}+\nabla \bar{h}^\prime(t) \nabla d(z(t)-\bar{z}(t),t)]^\prime \nabla\phi(z(t) + \nabla \bar{h}^\prime(t) d(z(t)-\bar{z}(t),t)) \\
& = & \nabla\phi(z(t)+\nabla \bar{h}^\prime(t) d(z(t)-\bar{z}(t),t)) \\
&& +\nabla d^\prime(z(t)-\bar{z}(t),t)\nabla \bar{h}(t) \nabla \phi(z(t)+\nabla \bar{h}^\prime(t) d(z(t)-\bar{z}(t),t)),
\end{eqnarray*}
where 
$$
\nabla d^\prime(z(t)-\bar{z}(t),t) \nabla \bar{h}(t) = \sum_{i=1}^{p} \nabla d_{i}(z(t)-\bar{z}(t),t) \nabla \bar{h}_{i}^\prime(t)~~\mbox{a.e.}~ t \in [0,T].
$$ 
Putting $z=\bar{z}$ and using (\ref{3.8}) results
$$
-\nabla \bar{h}^\prime(t)[\nabla \bar{h}(t)\nabla \bar{h}^\prime(t)]^{-1} \nabla \bar{h}(t) = \nabla d^\prime(0,t) \nabla \bar{h}(t) = \sum_{i=1}^{p}\nabla d_{i}(0,t) \nabla \bar{h}_{i}^\prime(t) ~~\mbox{a.e.}~ t \in [0,T].
$$
From the expression to $\nabla \varphi$ we obtain
\begin{eqnarray*}
&& \hspace{-0.5cm} \nabla^{2}\varphi(z(t),t) = \left\{ I_{n}+\nabla d^\prime(z(t)-\bar{z}(t),t) \nabla \bar{h}(t) \right\} \nabla^{2} \phi(z(t)+\nabla \bar{h}^\prime(t) d(z(t)-\bar{z}(t),t),t) \\
&& + \sum_{i=1}^{p} \nabla^{2}d_{i}(z(t)-\bar{z}(t),t)\nabla\bar{h}^\prime_{i}\nabla\phi(z(t)+\nabla \bar{h}^\prime(t) d(z(t)-\bar{z}(t),t),t) \\
&& + \sum_{i=1}^{p} \nabla d_{i}(z(t)-\bar{z}(t),t)\nabla\bar{h}^\prime_{i}(t) \nabla^{2}\phi(z(t)+\nabla \bar{h}^\prime(t) d(z(t)-\bar{z}(t),t),t) \\
&& + \sum_{i=1}^{p} \nabla d_{i}(z(t)-\bar{z}(t),t)\nabla\bar{h}^\prime_{i}(t) \nabla d^\prime(z(t)-\bar{z}(t),t) \nabla \bar{h}(t) \nabla^{2}\phi(z(t)+\nabla \bar{h}^\prime(t) d(z(t)-\bar{z}(t),t),t).
\end{eqnarray*}
Particularly for $z=\bar{z}$, taking $\gamma\in N$, it follows for a.e. $t \in [0,T]$ that
\begin{eqnarray*}
&& \hspace{-0.5cm} \gamma^\prime(t) \nabla^{2} \varphi(\bar{z}(t),t) \gamma(t) \\
&& \quad  = \gamma^\prime(t) \nabla^{2} \bar{\phi}(t) \gamma(t) - \gamma^\prime(t) \nabla\bar{h}^\prime(t) [\nabla \bar{h}(t) \nabla \bar{h}^\prime(t)]^{-1} \nabla \bar{h}(t) \nabla^{2} \bar{\phi}(t) \\
&& \qquad +\gamma^\prime(t) \left[ \sum_{i=1}^{p} \nabla^{2} d_{i}(0,t) \nabla \bar{h}^\prime_{i}(t) \right] \nabla \bar{\phi}(t) \gamma(t) \\
&& \qquad -\gamma^\prime(t) \nabla \bar{h}^\prime(t) [\nabla \bar{h}(t) \nabla \bar{h}^\prime(t)]^{-1} \nabla \bar{h}(t) \left[I_{n}+\nabla d^\prime(0,t)\nabla \bar{h}(t) \right] \nabla^{2} \bar{\phi}(t)\gamma(t),
\end{eqnarray*}
and integrating from $0$ to $T$ one has
\begin{eqnarray} \label{3.11} 
\int_{0}^{T}\gamma^\prime(t) \left\{ \nabla^{2} \bar{\phi}(t) + \left[\sum_{i=1}^{p} \nabla^{2} d_{i}(0,t) \nabla \bar{h}^\prime_{i}(t) \right] \nabla\bar{\phi}(t) \right \} \gamma(t) dt \leq 0.
\end{eqnarray}
On the other hand, once 
$$
\mu_{i}(\xi,d(\xi,t),t) = h_{i}(\bar{z}(t) + \xi + \nabla \bar{h}^\prime(t) d(\xi,t),t)~~\mbox{a.e.}~ t \in [0,T],~i\in I,
$$ 
we have
\begin{eqnarray*}
\nabla_\xi \mu_{i}(\xi,d(\xi,t),t) & = & [I_{n} + \nabla \bar{h}^\prime(t) \nabla d(\xi,t)]^\prime \nabla h_{i}(\bar{z}(t) + \xi + \nabla \bar{h}^\prime(t) d(\xi,t),t) \\
& = & \nabla h_{i}(\bar{z}(t) + \xi + \nabla \bar{h}^\prime(t) d(\xi,t),t) \\
&& + \nabla d^\prime(\xi,t) \nabla \bar{h}(t) \nabla h_{i}(\bar{z}(t) + \xi + \nabla \bar{h}^\prime(t) d(\xi,t),t),~i\in I.
\end{eqnarray*}
By (\ref{3.7}), for a.e. $t \in [0,T]$ and for each $i \in I$, we get
\begin{eqnarray*}
0 & = & \nabla^{2} \mu_{i}(\xi,d(\xi,t),t) \\
& = & [I_{n}+\nabla \bar{h}^\prime(t) \nabla d(\xi,t)]^\prime \nabla^{2} h_{i}(\bar{z}(t)+\xi+\nabla \bar{h}^\prime(t) d(\xi,t),t) \\ 
&& +\left[\sum_{j=1}^{p} \nabla^{2} d_{j}(\xi,t) \nabla \bar{h}^\prime_{j}(t) \right] \nabla h_{i}(\bar{z}(t)+\xi+\nabla \bar{h}^\prime(t) d(\xi,t),t) \\
&& +\left[\nabla d^\prime(\xi,t) \nabla \bar{h}(t) \right] \left[I_n + \nabla d^\prime(\xi,t) \nabla \bar{h}(t) \right] \nabla^{2} h_{i}(\bar{z}(t)+\xi+\nabla \bar{h}^\prime(t) d(\xi,t),t).
\end{eqnarray*}
We now put $\xi=0$ in the last expression, multiply it by $u_{i}(t)$ for a.e. $t \in [0,T]$, sum up from $1$ to $p$ and take the inner product with $\gamma \in N$, which results for a.e. $t \in [0,T]$, in
\begin{eqnarray*}
0 & = & \gamma^\prime(t) \left[ \sum_{i=1}^{p} u_{i}(t) \nabla^{2} \mu_{i}(0,0,t) \right] \gamma(t) \\
& = & \gamma^\prime(t) \left[ \sum_{i=1}^{p} u_{i}(t) \nabla^{2} \bar{h}_{i}(t) \right] \gamma(t) \\ 
& &-\gamma^\prime(t) \nabla \bar{h}^\prime(t) [\nabla \bar{h}(t) \nabla \bar{h}^\prime(t)]^{-1} \nabla \bar{h}(t) \left[ \sum_{i=1}^{p} u_{i}(t)\nabla^{2} \bar{h}_{i}(t) \right] \gamma(t) \\ 
&& + \gamma^\prime(t) \left[ \sum_{j=1}^{p} \nabla^{2} d_{j}(0,t) \nabla \bar{h}^\prime_{j}(t) \right] \left[ \sum_{i=1}^{p} u_{i}(t) \nabla \bar{h}_{i}(t)\right] \gamma(t )\\ 
&& - \gamma^\prime(t) \nabla \bar{h}^\prime(t)[\nabla \bar{h}(t) \nabla \bar{h}^\prime(t)]^{-1} \nabla \bar{h}(t) \left[ I_n + \nabla d^\prime(0,t)\nabla \bar{h}(t) \right] \left[ \sum_{i=1}^{p} u_{i}(t) \nabla^{2} \bar{h}_{i}(t) \right] \gamma(t).
\end{eqnarray*}
Integrating from $0$ to $T$ gives 
\begin{eqnarray}
&& \int_{0}^{T}\gamma^\prime(t) \left\{ \sum_{i=1}^{p} u_{i}(t) \nabla^{2} \bar{h}_{i}(t) + \left[ \sum_{j=1}^{p} \nabla^{2} d_{j}(0,t)\nabla \bar{h}^\prime_{j}(t) \right] \left[ \sum_{i=1} u_{i}(t) \nabla \bar{h}_{i}(t) \right] \right\} \gamma(t)~dt \nonumber \\  
&& \quad = 0. \label{3.12}
\end{eqnarray}
Adding (\ref{3.11}) and (\ref{3.12}) and using (\ref{3.3}), results in (\ref{3.4}).
\end{proof}

\section{KKT conditions for problems with equality and inequality constraints}

The general case is now tackled. Consider the problem (\ref{1.1}) with equality and inequality constraints. Given $\epsilon>0$ and $\bar{z} \in \Omega$, we assume that, in addition to (H1), the following hypotheses are valid:
\begin{itemize}
\item[(H5)] $h(z,\cdot)$ and $g(z,\cdot)$ are measurable for each $z$, $h(\cdot,t)$ and $g(\cdot,t)$ are continuously differentiable on $\bar{z}(t)+\epsilon \bar{B}~\mbox{a.e.}~ t \in [0,T]$.
\item[(H6)] There exists an increasing function $\bar{\theta} : (0,\infty) \rightarrow (0,\infty)$, $\bar{\theta}(s) \downarrow 0$ when $s \downarrow 0$, such that for all $\tilde{z}, z \in \bar{z}(t) + \epsilon \bar{B}$ and a.e. $t \in [0,T]$, 
$$
\| \nabla [h,g](\tilde{z},t) - \nabla [h,g](z,t)\| \leq \bar{\theta}(\|\tilde{z}-z\|).
$$ 
There exists $K_{1}>0$ such that for a.e. $t \in [0,T]$, 
$$
\| \nabla [h,g](\bar{z},t) \| \leq K_{1}.
$$
\item[(H7)] There exists $K>0$ such that  
$$
\det \{ \Upsilon(t) \Upsilon^{\prime}(t) \} \geq K,
$$ 
where 
$$
\Upsilon(t) = \left[ \begin{array}{cc} \nabla \bar{h}(t) & 0 \\ \nabla \bar{g}(t) & \mathrm{diag}\{-2\bar{w}_{j}(t)\}_{j\in J} \end{array}\right],
$$
and $\bar{w}_{j} = \sqrt{\bar{g}_{j}(t)}$ a.e. $t \in [0,T]$, $j\in J$.
\end{itemize}

\begin{remark} Assumption (H7), among other things, tell us that the vector set $\{ \nabla\bar{h}_{i}(t) \mid i\in I \} \cup \{ \nabla \bar{g}_{j}(t) \mid j\in I_{a}(t) \}$ is linearly independent for a.e. $t \in [0,T]$. 
\end{remark}

Next we present two examples referring the assumption (H7).

\begin{example} 
Consider the problem
\[
\begin{array}{ll} 
\mbox{maximize} & \displaystyle\int_{0}^{1}[-z_{1}^{2}(t)-z_{2}^{2}(t)]dt \\
\mbox{subject to} & h(z(t),t)=z_{1}(t)-z_{2}(t)=0 ~~ \mbox{a.e.} ~t \in [0,1],\\
& g_{1}(z(t),t)=z_{1}(t)+\frac{1}{2}z_{2}^{2}(t) \geq 0 ~~ \mbox{a.e.} ~ t \in [0,1],\\
& g_{2}(z(t),t)=z_{1}(t)z_{2}(t)+1 \geq 0 ~~ \mbox{a.e.} ~ t \in [0,1], 
\end{array}
\]
where $z=(z_{1},z_{2}) \in L_{2}^{\infty}[0,T]$ and $h,g_{1},g_{2}:\R^{2}\times [0,1] \rightarrow \R$. It is easy to see that $\bar{z}=(0,0)$ is an optimal solution and that $I_{a}(t)=\{ 1\}$ a.e. $t \in [0,T]$. Thus, the matrix in assumption (H7) is given by 
$$
\Upsilon(t)=\left[\begin{array}{ccc} 1 & -1 & 0 \\ 1 & 0 & 0 \\ 0 & 0 & -2\end{array}\right]~~\mbox{a.e.}~ t \in [0,T],
$$ 
which has full rank for a.e. $t \in [0,1]$. Note that $\{\nabla\bar{h}(t),\nabla\bar{g}_{1}(t)\}$ is linearly independent for a.e. $t \in [0,T]$.
\end{example}

\begin{example} 
Consider $h,g_{1},g_{2}:\R^{3}\times [0,1]\rightarrow\R$ and 
\[
\begin{array}{ll} 
\mbox{maximize} &\displaystyle\int_{0}^{1}[-(z_{1}(t)-1)^{2}-(z_{2}(t)-1)^{2}]dt \\
\mbox{subject to} & h(z(t),t)= -z_{1}^{2}(t)-z_{2}^{2}(t)+z_{3}(t)+1=0 ~~ \mbox{a.e.} ~t \in [0,1],\\
& g_{1}(z(t),t)= -2z_{1}z_{2}+4z_{2}+z_{3}-3\geq 0 ~~ \mbox{a.e.} ~t \in [0,1],\\
& g_{2}(z(t),t)= -z_{1}(t)+\frac{1}{2}z_{3}(t)+\frac{1}{2} \geq 0 ~~ \mbox{a.e.} ~t \in [0,1].
\end{array}
\]
The feasible point $\bar{z}=(1,1,1)$ is an optimal solution for this problem. Note that, for a.e. $t \in [0,1]$, $I_{a}(t) = \{ 1,2 \}$ and 
$$
\Upsilon(t)=\left[\begin{array}{ccccc}-2 & -2 & 1 & 0 & 0 \\ -2 & 2 & 1 & 0 & 0 \\ -1 & 0 & \frac{1}{2} & 0 & 0 \end{array}\right]~~\mbox{a.e.}~t \in [0,1].
$$
Provided $\mathrm{rank}(\Upsilon(t)) = 2$ a.e. $t \in [0,1]$, (H7) is not valid, though $\bar{z}$ is an optimal solution.
\end{example}

The lemma below will be used in the proof of the main result of this section.

\begin{lemma}\label{lema 4.1} 
Let $k\in J$ be arbitrary and $D \subset [0,T]$ be a subset of positive measure such that $k \in I_{a}(t)$ for all $t \in D$. If assumption (H7) holds true, then there exists $\gamma \in L_{\infty}([0,T];\R^{n})$ such that, for all $t\in D$, one has
\begin{equation}\label{4.1}
\nabla \bar{h}^\prime_{i}(t) \gamma(t) = 0,~i\in I,~~\nabla \bar{g}^\prime_{j}(t) \gamma(t) = 0,~j\in I_{a}(t) \setminus \{ k \},~~\nabla\bar{g}^\prime_{k}(t) \gamma(t) > 0.
\end{equation}
\begin{proof} 
If the components of $g$ are permuted in such a way that the active constraints come first, then $\Upsilon(t)$ in (H7) can be rewritten as  
$$
\Upsilon(t) = \left[\begin{array}{cc}
\nabla\bar{h}(t) & 0  \\ \nabla\bar{g}^{I_{a}(t)}(t) & 0  \\ 
\nabla\bar{g}^{I_{c}(t)}(t) &  \Lambda(t)
\end{array}\right]~~\mbox{a.e.}~ t \in [0,T],
$$ 
where $\Lambda(t) = \left[ \begin{array}{cc} 0 & \Lambda_{1}(t) \end{array} \right]$, $0 \in \R^{q_{c}(t) \times q_{a}(t)}$ and $\Lambda_{1}(t) = \mathrm{diag}\{-2\bar{w}_{j}(t)\}$, $j\in I_{c}(t)$. By (H7), the matrix $\Upsilon(t)$ has full rank for almost every $t\in [0,T]$. As $k\in I_{a}(t)$, for all $t\in D$, let $b \in L_{\infty}([0,T];\R^{n+m})$ be given as
$$
b_{j}(t) = \left\{ \begin{array}{ll}
0, & \mbox{if} ~ j \in I \cup J \setminus \{ k \}, ~ t \in D,\\ 
1,& \mbox{if} ~ j=k, ~ t \in D, \\
0 & \mbox{if} ~ t \in [0,T] \setminus D.
\end{array} \right.
$$
Then the system 
\begin{equation} \label{4.2} 
\Upsilon(t) \gamma(t) = b(t),
\end{equation}
is consistent for almost every $t\in [0,T]$, with solution $\gamma = (\gamma,\gamma_{1}) \in L_{\infty}([0,T];\R^{n+m})$ given by 
$$
\gamma(t) = \Upsilon'(t)[\Upsilon(t)\Upsilon'(t)]^{-1}b(t)~\mbox{a.e.} ~ t \in [0,T].
$$
Particularly, (\ref{4.2}) holds for all $t\in D$, so that, for all $t\in D$, there exists $\gamma \in L_{n}^{\infty}[0,T]$ such that
$$
\nabla \bar{h}^\prime_{i}(t)\gamma(t) = 0, ~ i \in I, ~~ \nabla \bar{g}^\prime_{j}(t) \gamma(t) = 0, ~ j \in I_{a}(t) \setminus \{ k \} ~~\mbox{and}~~\nabla \bar{g}^\prime_{k}(t) \gamma(t) = 1 > 0.
$$ 
\end{proof}
\end{lemma}

Next, the Karush-Kuhn-Tucker type optimality conditions are obtained for the general case.

\begin{theorem}\label{TKKT-IL-PTC} 
Let $\bar{z} \in \Omega$ be a local optimal solution for the problem (\ref{1.1}). Suppose that (H1), (H5)-(H7) do hold and that $g(\bar{z}(\cdot),\cdot)$ is bounded in $[0,T]$. Then there exists $(u,v) \in L_{\infty}([0,T];\R^{p}\times\R^{m})$ such that for a.e $t \in [0,T]$ one has
\begin{eqnarray}
&& \nabla \bar{\phi}(t) + \sum_{i=1}^{p} u_{i}(t) \nabla \bar{h}_{i}(t) + \sum_{j=1}^{m} v_{j}(t) \nabla \bar{g}_{j}(t) = 0, \label{4.3} \\
&& v(t)\geq 0, \\
&& v_{j}(t)\bar{g}_{j}(t) = 0, ~ j \in J. \label{4.4}
\end{eqnarray}
Moreover, 
\begin{equation} \label{4.5} 
\int_{0}^{T} \gamma^\prime(t) \lbrace \nabla^{2} \bar{\phi}(t) + \sum_{i=1}^{p} u_{i}(t) \nabla^{2} \bar{h}_{i}(t) + \sum_{j=1}^{m} v_{j}(t) \nabla^{2} \bar{g}_{j}(t) \rbrace \gamma(t) dt \leq 0
\end{equation}
for all $\gamma \in \bar{N}$, where $\bar{N}$ is given by
$$
\bar{N} = \{ \gamma \in L_{\infty}([0,T];\R^{n}) ~\mid~ \nabla \bar{h}(t)\gamma(t)=0,~~\nabla \bar{g}'_{j}(t)\gamma(t)=0,~j\in I_{a}(t), ~\mbox{a.e.} ~ t \in [0,T] \}.
$$
\end{theorem}
\begin{proof} 
Let $w : [0,T] \rightarrow\R^{m}$ be a measurable function and consider the auxiliary problem below
\begin{equation}\label{4.6}
\begin{array}{ll} 
\mbox{maximize} & \tilde{P}(z,w) = \displaystyle\int_{0}^{T}\phi(z(t),t)dt \\ 
\mbox{subject to} & h(z(t),t)=0~~\mbox{a.e.}~ t \in [0,T], \\
& g(z(t),t)-w^{2}(t)=0~~\mbox{a.e.}~ t\in [0,T],
\end{array}
\end{equation}
where $$
w^{2}(t) = \left[ \begin{array}{c} w^{2}_{1}(t) \\ w^{2}_{2}(t) \\ \vdots \\ w^{2}_{m}(t) \end{array} \right] ~~ \mbox{a.e.}~ t \in [0,T].
$$
We proceed in several steps.

\noindent\textbf{STEP 1:} If $\bar{z}$ is local optimal solution for the problem (\ref{1.1}), then $(\bar{z},\bar{w})$ is local optimal solution for the problem (\ref{4.6}), where 
$$
\bar{w}_{j}(t)=\sqrt{\bar{g}_{j}(t)}~~\mbox{a.e.}~ t \in [0,T],~ j\in J.
$$ 
Indeed, if $\bar{z}$ is a solution of (\ref{1.1}) in a $\epsilon \in (0,1)$ neighborhood, suppose that for all $0<\delta<\epsilon$, there exists $(\tilde{z},\tilde{w}) \in (\bar{z},\bar{w}) + \delta B$ with $\tilde{P}(\tilde{z},\tilde{w}) >\tilde{P}(\bar{z},\bar{w})$. Noticing that 
$$
h(\tilde{z}(t),t)=0~~\mbox{and}~~g(\tilde{z}(t),t)=\tilde{w}^{2}(t)\geq 0~~\mbox{a.e.}~ t \in [0,T],
$$ 
we see that $\tilde{z}$ is feasible for the problem (\ref{1.1}) and
$$
P(\tilde{z}) = \int_{0}^{T} \phi(\tilde{z}(t),t)dt = \tilde{P}(\tilde{z},\tilde{w}) > \tilde{P}(\bar{z},\bar{w}) = \int_{0}^{T}\phi(z(t),t)dt = P(\bar{z}),
$$ 
contradicting the local optimality of $\bar{z}$ for (\ref{1.1}).

\noindent\textbf{STEP 2:} Define
$$
\Psi(z,w,t) = \left[ \begin{array}{c} h(z,t) \\ g(z,t) - w^{2} \end{array} \right] \in \R^{p+m}.
$$ 
We will verify that the auxiliary problem (\ref{4.6}) satisfies the conditions (H1)-(H4) with $\Psi$ and $(z,w)$ playing the role of $h$ and $z$, respectively. The assumptions (H1) and (H2) are immediate. Considering $(\tilde{z},\tilde{w}),~(z,w)\in (\bar{z}(t),\bar{w}(t))\in\epsilon\bar{B}$ we have for a.e. $t \in [0,T]$ that
\begin{eqnarray*}
&& \hspace{-0.5cm} \| \nabla \Psi(\tilde{z},\tilde{w},t) - \nabla \Psi(z,w,t) \| \\
&& \quad =  \| [\nabla_{z} \Psi(\tilde{z},\tilde{w},t) - \nabla_{z} \Psi(z,w,t) ~~ \nabla_{w} \Psi(\tilde{z},\tilde{w},t) - \nabla_{w} \Psi(z,w,t)] \| \\
&& \quad \leq \| \nabla [h,g](\tilde{z},t) - \nabla[h,g](z,t) \| + \| (-2)\mathrm{diag} \{ \tilde{w}_{i}(t) - w_{i}(t) \} \| \\
&& \quad = \| \nabla [h,g](\tilde{z},t) - \nabla[h,g](z,t) \| + 2\| (\tilde{w}(t)-w(t)) \| \\ 
&& \quad \leq \bar{ \theta}(\|\tilde{z}-z\|) + 2 \| \tilde{w}-w \| = \bar{\theta}(\|\tilde{z}-z\|)+2\|\tilde{w}-w\| \\ 
&& \quad \leq \bar{\theta}(\|(\tilde{z}-z,\tilde{w}-w)\|) + 2\| (\tilde{z}-z,\tilde{w}-w) \| = \tilde{\theta}(\|(\tilde{z},z)-(\tilde{w},w)\|),
\end{eqnarray*}
where $\tilde{\theta} : (0,\infty) \rightarrow (0,\infty)$ is given by $\tilde{\theta}(s) = \bar{\theta}(s)+2s$. $\tilde{\theta}$ is an increasing function and when $s\downarrow 0$, $\tilde{\theta}(s)=\bar{\theta}(s)+2s\downarrow 0$. Also,
\begin{eqnarray*}
\| \nabla \Psi(\bar{z}(t),\bar{w}(t),t) \| & \leq & \| \nabla_{z} \Psi(\bar{z}(t),\bar{w}(t),t) \| + \| \nabla_{w} \Psi(\bar{z}(t),\bar{w}(t),t) \| \\
& = & \| \nabla[h,g](\bar{z}(t),t) \| + 2\| \mathrm{diag}\{\bar{w}_{i}(t)\} \| \leq K_{0},
\end{eqnarray*} 
where $K_{0}=K_{1}+\|\bar{w}\|^{\infty}_{m}$ comes from (H6) and from the assumption that $\bar{w}(t)=g(\bar{z}(t),t)$ is uniformly bounded in $[0,T]$, verifying (H3). Finally, as
$$
\nabla \Psi(\bar{z}(t),t) = \Upsilon(t)~~\mbox{a.e.}~ t \in [0,T],
$$
the assumption (H7) implies (H4).

\noindent\textbf{STEP 3:} By Theorem \ref{TKKT-EC}, there exists $(u,v) \in L_{\infty}([0,T];\R^{p}\times\R^{m})$ such that 
$$
\left[ \begin{array}{c} \nabla \bar{\phi}(t) \\ 0 \end{array} \right] + 
\nabla \Psi^\prime(\bar{z}(t),t) \left[ \begin{array}{c} u(t) \\ v(t) \end{array} \right] =0 ~~ \mbox{a.e.}~ t \in [0,T],
$$ 
which implies in
$$
\nabla \bar{\phi}(t) + \sum_{i=1}^{p} u_{i}(t) \nabla \bar{h}_{i}(t) + \sum_{j=1}^{m} v_{j}(t) \nabla \bar{g}_{j}(t) = 0~~\mbox{a.e.}~ t \in [0,T]
$$
and
$$
\bar{w}_{j}(t)v_{j}(t) = 0 ~~ \mbox{a.e.}~ t \in [0,T] \Rightarrow \bar{g}_{j}(t) v_{j}(t) = 0 ~~ \mbox{a.e.}~ t \in [0,T], ~ j \in J,
$$ 
resulting in (\ref{4.3}) and (\ref{4.4}).  

We will show now that $v_{j}(t) \geq 0$ for a.e. $t \in [0,T]$, $j\in I$. Suppose that for any $k \in J$, there exists a positive measure subset $D \subset [0,T]$ such that $v_{k}(t) < 0$, for all $t\in D$. Note that, by (\ref{4.4}), $k \in I_{a}(t)$ for all $t \in D$. By Lemma \ref{lema 4.1}, there exists $\gamma \in L_{\infty}([0,T];\R^{n})$ such that (\ref{4.1}) holds. Defining $\tilde{\gamma} \in L_{\infty}([0,T];\R^{n})$ as
$$
\tilde{\gamma}(t) = \Biggl\{ \begin{array}{l} \gamma(t), ~~ \mbox{if} ~ t\in D \\ 0, ~~ \mbox{otherwise}, \end{array}
$$
and using (\ref{4.1}) and (\ref{4.3}), one has
$$
\nabla \bar{\phi}^\prime(t) \tilde{\gamma}(t) + \sum_{i=1}^{p} u_{i}(t) \nabla\bar{h}^\prime_{i}(t) \tilde{\gamma}(t) + \sum_{j=1}^{m} v_{j}(t) \nabla\bar{g}^\prime_{j}(t) \tilde{\gamma}(t) = 0
$$
$$
\Rightarrow \nabla \bar{\phi}^\prime(t) \tilde{\gamma}(t) = \Biggl\{ \begin{array}{l} -v_{k}(t) \nabla \bar{g}^\prime_{k}(t) \gamma(t)>0, ~~ \mbox{if} ~ t \in D, \\ 0, ~~ \mbox{otherwise}. \end{array}
$$
Then 
$$
\int_{0}^{T} \nabla \bar{\phi}^\prime(t) \tilde{\gamma}(t)dt = \int_{D}^{T} \nabla \bar{\phi}^\prime(t) \tilde{\gamma}(t)dt + \int_{[0,T] \setminus D}\nabla \bar{\phi}^\prime(t) \tilde{\gamma}(t)dt > 0
$$ 
and by Proposition \ref{Reiland} we have that there exists $\sigma>0$ such that $P(\bar{z}+\tau\tilde{\gamma}) > P(\bar{z})$ for $0 < \tau \leq \sigma$, contradicting the fact that $\bar{z}$ is a local optimal solution for the problem (\ref{1.1}). Therefore, $v_{j}(t)\geq 0$ a.e. $t \in [0,T]$, $j\in J$.

\noindent{\bf STEP 4:} We will verify the second order condition (\ref{4.5}). Let us denote, for a.e. $t \in [0,T]$,
$$
L(z(t),w(t),t) = \phi(z(t),t) + \sum_{i=1}^{p} u_{i}(t) h_{i}(z(t),t) + \sum_{j=1}^{m} v_{j}(t) [g_{j}(z(t),t)-w_{j}^{2}(t)],
$$ 
where $u$ and $v$ are the multipliers previously obtained at \textbf{STEP 3}. Then,
\begin{eqnarray*}
\nabla L(z(t),w(t),t) & = & \left[ \begin{array}{c} \nabla_{z}L(z(t),w(t),t) \\ \nabla_{w}L(z(t),w(t),t) \end{array} \right] \\
& = & \left[ \begin{array}{c}
\nabla \phi(z(t),t) + \displaystyle\sum_{i=1}^{p} u_{i}(t) \nabla h_{i}(z(t),t) + \displaystyle\sum_{j=1}^{m} v_{j}(t) \nabla g_{j}(z(t),t) \\ 
-2w_{1}(t)v_{1}(t) \\ \vdots \\ -2w_{m}(t)v_{m}(t)
\end{array} \right],
\end{eqnarray*}
and
$$
\nabla^{2}L(z(t),w(t),t) = \left[ \begin{array}{cc} \nabla_{zz}L(z(t),w(t),t) & 0 \\ 0 & \mathrm{diag}\{-2v_{j}(t)\}_{j=1}^{m} \end{array} \right],
$$ 
where 
$$
\nabla_{zz}L(z(t),w(t),t) = \nabla^{2} \phi(z(t),t) + \sum_{i=1}^{p} u_{i}(t) \nabla^{2} h_{i}(z(t),t) + \sum_{j=1}^{m}v_{j}(t)\nabla^{2} g_{j}(z(t),t).
$$
By Theorem \ref{TKKT-EC}, we have that 
\begin{equation} \label{4.7} 
\int_{0}^{T}(\gamma(t),\nu(t))^\prime\nabla^{2} L(\bar{z}(t),\bar{w}(t),t)(\gamma(t),\nu(t))dt \leq 0,
\end{equation}
for all $(\gamma,\nu) \in L_{\infty}([0,T];\R^{n}\times \R^{m})$ satisfying
\begin{equation} \label{4.8} 
\nabla h(\bar{z}(t),t)\gamma(t) = 0\hspace{0.5cm} \mbox{and} \hspace{0.5cm} \nabla g^\prime_{j}(\bar{z}(t),t)\gamma(t)-2\bar{w}_{j}(t)\nu_{j}(t)=0,~~j\in J,
\end{equation}
for a.e. $t \in [0,T]$. For all $\gamma \in \bar{N}$, consider $\nu$ defined, for a.e. $t \in [0,T]$, as 
$$
\nu_{j}(t) = \Biggl\{ \begin{array}{l} 0,~~\mbox{if}~j\in I_{a}(t), \\ \displaystyle\frac{\nabla g^\prime_{j}(\bar{z}(t),t)\gamma(t)}{2\bar{w}_{j}(t)},~~\mbox{if}~j\in I_{c}(t) \end{array}.
$$ 
Then, note that $(\gamma,\nu)$ satisfies (\ref{4.8}), $\bar{w}_{j}(t) \nu_{j}(t)=0~~\mbox{a.e.}~ t \in [0,T],~\forall~j\in I_{a}(t)$, and by (\ref{4.4}) we have that $v_{j}(t)=0~\mbox{a.e.} ~ t \in [0,T],~\forall~j\in I_{c}(t)$. Thus, 
$$
v_{j}(t)\nu_{j}(t)=0~~\mbox{a.e.} ~ t \in [0,T], ~ j \in J.
$$ 
Replacing $(\gamma,\nu)$ in (\ref{4.7}), we obtain 
$$
\int_{0}^{T} \gamma^\prime(t) \nabla^{2}_{zz} L(\bar{z}(t),\bar{w}(t),t)\gamma(t)dt \leq 0,
$$ 
with arbitrary $\gamma \in \bar{N}$, implying in (\ref{4.5}).  
\end{proof}

\section*{Acknowledgments}

V.A. de Oliveira was partially supported by grants 2013/07375-0 and 2016/03540-4, São Paulo Research Foundation (FAPESP), and by grants 457785/2014-4 and 310955/2015-7, National Council for Scientific and Technological Development (CNPq).

\bibliographystyle{plain}

\end{document}